\documentclass[11pt,a4paper]{article}

\usepackage[english]{babel}
\usepackage[utf8]{inputenc}
\usepackage{amssymb}
\usepackage{amsfonts}
\usepackage{amsmath}
\usepackage{amsthm}
\usepackage{graphicx}
\usepackage[colorinlistoftodos]{todonotes}
\usepackage{authblk}
\usepackage{cite} 
\usepackage[labelfont=bf,labelsep=space]{caption}
\usepackage[framed,numbered,autolinebreaks,useliterate]{mcode}
\usepackage{subfig}
\usepackage{multirow}

\usepackage{hyperref}

\newtheorem{thm}{Theorem}
\newtheorem{cor}{Corollary}

\begin{document}
\vspace*{0.35in}

\begin{flushleft}
{\Large
\textbf\newline{Modeling Joint Improvisation between Human and Virtual Players in the Mirror Game}
}
\newline
\\
Chao Zhai \textsuperscript{1},
Francesco Alderisio \textsuperscript{1},
Piotr S\l{}owi\'{n}ski \textsuperscript{2},
Krasimira Tsaneva-Atanasova \textsuperscript{2},
Mario di Bernardo \textsuperscript{1,3,*}
\\
\bigskip
\bf{1} Department of Engineering Mathematics, University of Bristol, BS8 1UB Bristol, United Kingdom \\
\bf{2} College of Engineering, Mathematics and Physical Sciences, University of Exeter, EX4 4QF Exeter, United Kingdom \\
\bf{3} Department of Electrical Engineering and Information Technology, University of Naples Federico II, 80125 Naples, Italy
\bigskip

%
%





* Email: m.dibernardo@bristol.ac.uk

\end{flushleft}
\section*{Abstract}
Joint improvisation is observed to emerge spontaneously among humans performing joint action tasks, and has been associated with high levels of movement synchrony and enhanced sense of social bonding. 
Exploring the underlying cognitive and neural mechanisms behind the emergence of joint improvisation is an open research challenge.
This paper investigates the emergence of jointly improvised movements between two participants in the mirror game, a paradigmatic joint task example. A theoretical model based on observations and analysis of experimental data is proposed to capture the main features of their interaction. A set of experiments is  carried out to test and validate the model ability to reproduce the experimental observations. Then, the model is used to drive a computer avatar  able to improvise joint motion with a human participant in real time. Finally, a convergence analysis of the proposed model is  carried out to confirm its ability to reproduce the emergence of joint movement between the participants.
\section*{Author Summary}
The aim of this paper is to develop a mathematical framework to explain and capture the emergence of joint improvisation between two individuals playing the mirror game, and generate a model-driven virtual player able to interact and coordinate its motion with that of a human subject. We take the mirror game as a paradigmatic example of joint action. In so doing we propose a theoretical model of joint improvisation in the mirror game. We show that such a model is able to account for the onset of movement synchronization in joint-action tasks between two players and to guarantee that a model-driven computer avatar interacts successfully in real-time with a human player generating, for the first time in the literature, new jointly improvised movements.

\section{Introduction}
\label{sec:intro}
Human social interactions give rise to a variety of self-organized and emergent motor behaviors \cite{Kelso1997,Schmidt2008,Schmidt2014,pnas14}. A typical example is joint improvisation between two humans performing some task together,  as engaging in a conversation or performing social dancing~\cite{john12,saw01,noy15}.  Experimental results suggest that interpersonal synchrony in joint actions unconsciously fosters social rapport and promotes a sense of affinity between two individuals~\cite{lak11,wil09}.

To investigate the mechanisms behind the emergence of social interaction between two individuals, the Human Dynamic Clamp (HDC) paradigm has been recently proposed in~\cite{kel09,pnas14,kostrubiec15} where a model-driven avatar (or virtual player) replaces one of the two humans. In so doing, the features of the motion of the virtual player can be manipulated in order to understand whether and how the interaction with the human subject is affected.

Here we present a new model to explain and analyze the emergence of joint improvisation between two players. We focus on the mirror game, a paradigm of joint human interaction which was recently proposed in ~\cite{noy11,hart14}. Contrary to previous approaches where models were typically used to generate simple oscillatory motion or to reproduce the motion of a human subject \cite{pnas14}, we present a model able to capture the complex movements generated by human subjects playing the game and generate new motion.
Furthermore, we show that the model  can capture and reproduce the essential kinematic features of the movement of a reference human subject as encoded by the Individual Motor Signature recently introduced in ~\cite{piotr,piotr15}, opening the possibility of testing  {\em in-silico}  the interaction between different individuals in a number of different configurations.
We also use our model to design  an enhanced model-driven avatar (or virtual player) able not only to interact with a human subject but also to generate jointly improvised movement with him/her. 

The predictions of the model and its ability to drive a computer avatar in real-time are tested and validated via an exhaustive experimental investigation, accompanied by a mathematical analysis of its convergence.

\section{Methods}

\subsection{Mirror game}\label{sec:mg}

As a simple yet effective paradigm to study interpersonal coordination between two individuals we use the mirror game~\cite{noy11}. Specifically,  as shown in Fig~\ref{fig1}, two players facing each other are asked to coordinate the motion of two balls mounted on two respective parallel strings. The players can be asked to play in a leader-follower condition (LF), where one is instructed to follow the motion of the other, or in a joint improvisation condition (JI), where they are instructed to imitate each other, create synchronized and interesting motions and enjoy playing together, without any designation of leader and follower. 

\begin{figure}
 \includegraphics{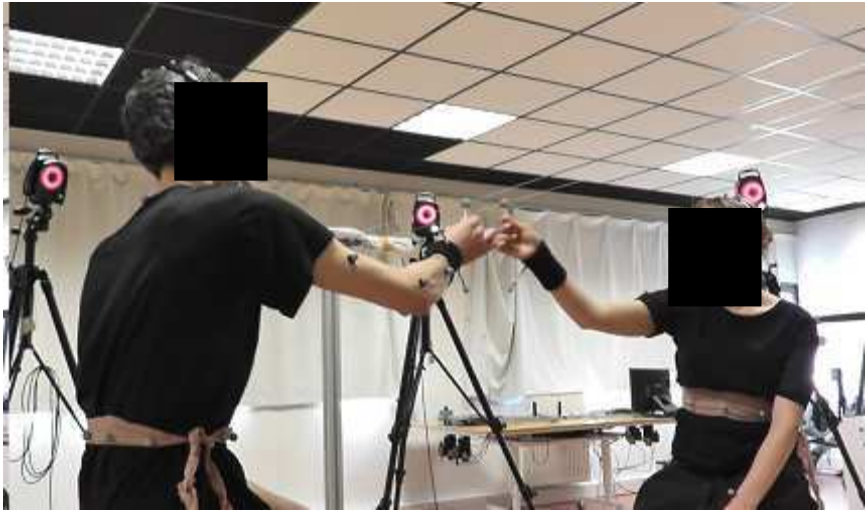}\centering
\caption{\label{fig1}{\bf Mirror game between two human players at University of Montpellier, France.} Two participants face each other and are asked to perform synchronized motion by moving two balls along a string to which they are respectively attached.}
\centering
\end{figure}

\subsection{Individual motor signature}

As suggested in \cite{piotr,piotr15}, the motion of each individual in the mirror game is characterized by different kinematic features that can be accounted for by examining the velocity profile of the player's motion during the game. The velocity profile comprises a velocity frequency distribution, termed in \cite{piotr} as Individual Motor Signature (IMS), and can be used to classify and distinguish the movement generated by different human subjects.
To acquire the IMS of a human subject, we asked him/her to play the mirror game in a Solo condition, i.e. in the absence of the other participant. In this condition, the human subject was asked to generate interesting complex motion for 60 seconds. 
The position time series were recorded during the experiment, and were next used to estimate the velocity PDF of the player's motion.

\subsection{Experimental set-ups}
\label{sec:expSetup}

We use two experimental set-ups to carry out experiments. 
\begin{itemize}

\item Set-up 1 (shown in Fig~\ref{fig1}) consists of two parallel strings, with a ball that can slide on each. Two human players are asked to move their own ball back and forth along the strings, respectively, while seated. The position of the balls are detected by cameras disposed around the two participants. Details of the set-up can be found in \cite{piotr,piotr15}.

\item Set-up 2 (schematically shown in Fig~\ref{fig2}) employs a cheap leap motion controller  to detect the fingertip position of the human player (HP).
\begin{figure}
 \centering
 \includegraphics[width=0.8\textwidth]{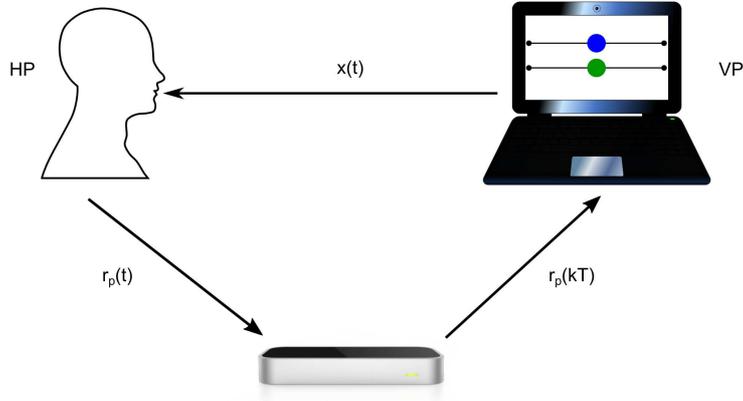}
\caption{\label{fig2}{\bf Experimental setup of the mirror game between a HP and a VP.} The position of the human fingertip $r_p(t)$ is detected by a leap motion controller, and the sampled position $r_p(kT)$ is sent to the computer, while the position $x(t)$ of the VP is generated by implementing the numerical algorithm of the single model. Two circles are shown on the computer screen, which correspond to the end effectors' position of the HP (blue circle) and the VP (green circle), respectively.}
\end{figure}
Both the leap motion controller and a laptop computer (employed to implement and run the theoretical model driving the computer avatar) are placed on a table. The HP is asked to wave his/her own index finger horizontally over the leap motion controller with a horizontal range of around $60cm$. The fingertip position of the HP is mapped into the interval $[-0.5,0.5]$ and visualized as a blue solid circle on the computer screen. A  green solid circle, whose movement is computed from the model presented in this paper, is also visualized to represent the position of the virtual player (VP).
The advantage of using this simple set-up, consisting of cheap off-the-shelf elements, is its accessibility and ease of implementation.
\end{itemize}

\subsection{Data and evaluation metrics}

To assess the level of coordination between the players we use the following metrics.
  
\begin{itemize}
  \item {\it Temporal correspondence.} The root mean square (RMS) of the normalized position error between the two players defined as  
  $$
e_p=\frac{1}{L}\sqrt{\frac{1}{n}\sum_{k=1}^{n}(x_{1,k}-x_{2,k})^2}
  $$
 is used to describe their tracking accuracy (temporal correspondence) during the game.  Here  $L$ refers to the range of admissible positions (e.g. the length of the strings in the set-up shown in Fig~\ref{fig1} or the range of motion detected by the leap motion controller), $n$ is the number of sampling steps in the simulation, and $x_{1,k}$ and $x_{2,k}$ denote the positions of two participants at the $k$-th sampling step, respectively.
  
  \item {\it Distance between IMSs.} The earth mover's distance (EMD) is used to quantify differences between the velocity PDFs of the two players' motion (e.g. to assess how similar/dissimilar their IMSs are). It is a proper metric in the space of PDFs~\cite{levi01,piotr15} and is computed as follows:
  $$
\eta(p_1,p_2)=\int_{Z}|CDF_{p_1}(z)-CDF_{p_2}(z)|dz
  $$
 where $Z$ denotes the integration domain, and $CDF_{p_i}(z)$ denotes the cumulative distribution function of the distribution $p_i,i\in\{1,2\}$. Furthermore, we normalize the EMDs with the maximal $\eta_{max}$ given by the length of the integration domain $\eta_{max} = |Z|$.
  
\item {\it Relative phase distribution.} The PDF of the relative phase $\phi_{12}$ between the players motion (estimated by means of wavelet coherence~\cite{Grinsted2004}) is used to check the directionality of the interaction during the game (i.e. if one player is leading or following the other).
\end{itemize}

\subsection{Human benchmark} \label{sec:HP_HP}
To establish a benchmark dataset to compare with model predictions, we obtained data from 8 different human dyads. Data from each dyad contains 3 solo trials for each human participant and 3 joint trials between them in JI condition. Description of all the available data can be found in Section \ref{sec:availabledata}.

A representative example of the data collected in mirror game between two HPs (Dyad 1, HPs JI trial 3) is shown in Fig~\ref{fig3}.
\begin{figure}
\centering
\includegraphics{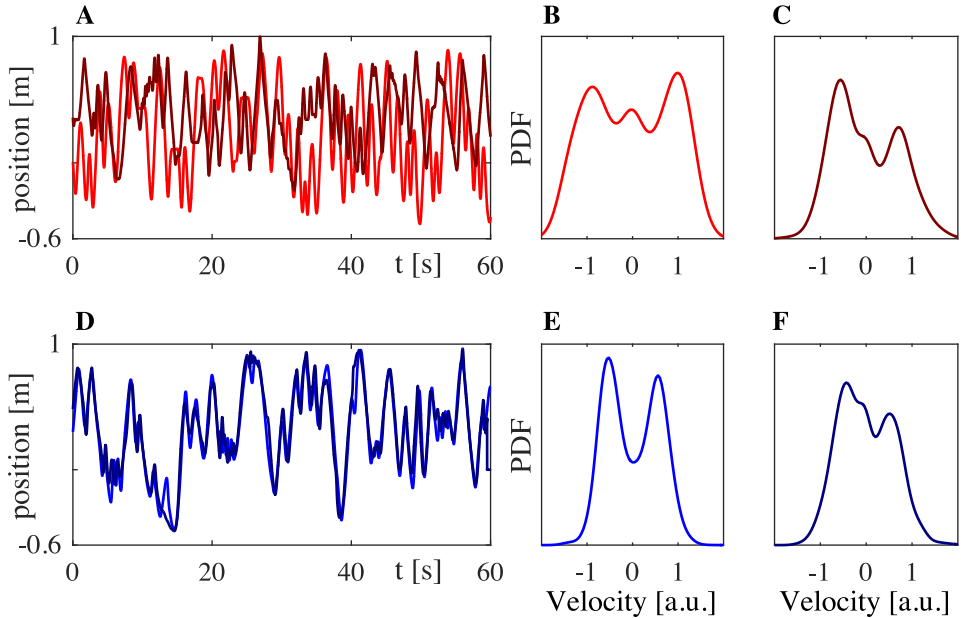}
\caption{\label{fig3}{\bf Example of position time series and velocity PDFs from experimental results for Dyad 1.} A: position time series for the solo trials of the HPs. B-C: PDFs of velocity corresponding to the two players. D: position time-series of the HPs from the JI trial. E-F: PDFs of velocity of the two HPs in JI condition.}
\end{figure}
In particular, Fig~\ref{fig3}A shows the trajectories of the first solo trial of  HP$_1$ (light red) and the third solo trial of HP$_2$ (dark red), while Fig~\ref{fig3}D shows the trajectories of the two HPs interacting in JI condition (light and dark blue). Additionally, we show the velocity PDFs estimated from each of the respective position time series: panels B-C from Solo and panels E-F from JI between the two HPs.

In order to visualize the relations between the distributions, we used multidimensional scaling (MDS), a data mining and visualization technique~\cite{MDS,piotr15} that uses distances/ dissimilarities between objects to represent them as points in a geometric space.

An example of such a geometric representation of the relations between the PDFs of the players' velocities is shown in Fig~\ref{fig4}A. Different velocity PDFs are denoted with different markers: $\sigma_i$ (red dots) indicates the motor signature of the $i$-th human player when playing solo; $\mu_{i}$ (blue dots) indicates the velocity PDFs of the $i$-th human player during runs of the mirror game with the other player.

\begin{figure}
\centering
\includegraphics{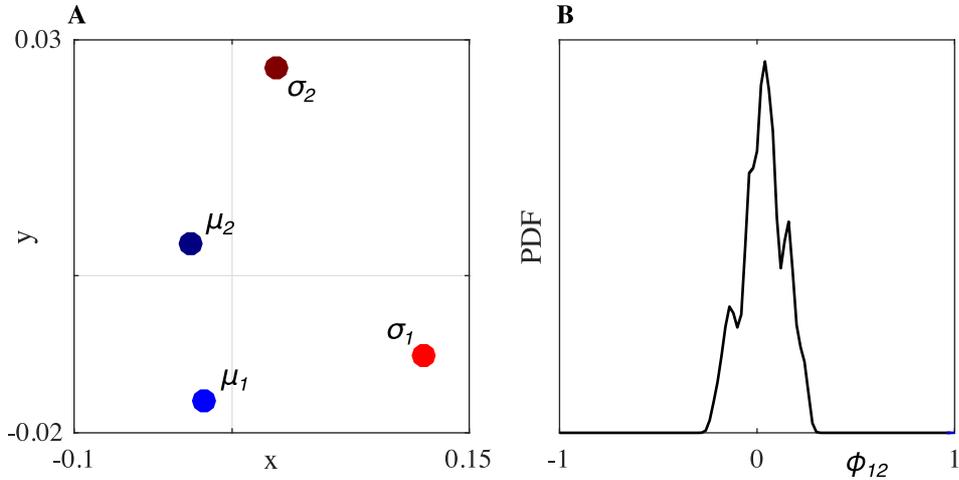}
\caption{\label{fig4}{\bf Example of data analysis for Dyad 1.} A: visualization of the relations between kinematics of two human players obtained by means of MDS. Red dots $\sigma_{1,2}$ indicate the signatures of the two respective HPs. Blue dots labelled as $\mu_{1,2}$ indicate the velocity profile of the motion of the HPs in the JI trial. B: PDF of relative phase $\phi_{12}$ between the two HPs.}
\end{figure}

In agreement with previous studies \cite{hart14}, we found that the velocity PDFs of the two HPs (Fig~\ref{fig4}A, light and dark blue dots) move away from their respective motor signatures (Fig~\ref{fig4}A, light and dark red dots) when they are improvising together because of mutual imitation, adaptation and synchronization that results in their velocity PDFs moving towards each other during the game. The values of the distances between the velocity PDFs depicted in Fig~\ref{fig4}A were computed to be $\eta(\sigma_1,\mu_{1})=0.102$, $\eta(\sigma_2,\mu_{2})=0.052$ and $\eta(\mu_{1},\mu_{2})=0.030$.

Finally, in Fig~\ref{fig4}B, we plot the distribution of the  relative phase between the two players. We found that it is quite broad and centered around 0, indicating that neither of the two players was clearly leading the interaction during the game. The players, as instructed, were generating jointly improvised motion for most of the time.

Similar results were obtained for all the trials of every dyad (Fig~\ref{fig5}). Each panel of Fig~\ref{fig5} corresponds to a single dyad. For each dyad, three PDFs of relative phase from the three respective JI trials are shown with different scales of blue. It is possible to appreciate how all the PDFs are quite broad and centered around 0. The only exception is Dyad 7 where the maximum of the relative phase PDF is shifted on the right for all the three trials indicating that, despite the instruction given to the two participants, one player consistently ended up leading the game (Fig~\ref{fig5}G).

\begin{figure}
\includegraphics[width=\textwidth]{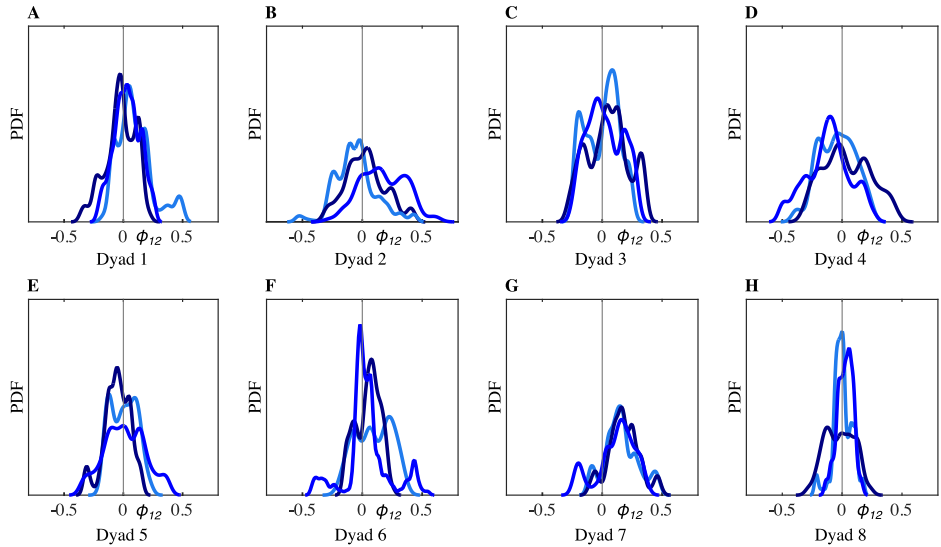}\centering
\caption{\label{fig5}{\bf PDFs of relative phase between HPs in all the dyads.} Each panel corresponds to a single dyad. For each dyad, three HP$_1$-HP$_2$ trials are shown with different scales of blue. The relative phase $\phi_{12}$ between players is estimated by means of wavelet coherence (with 1Hz cut-off frequency). PDFs are estimated from histograms with a kernel density estimation.}
\end{figure}

\section{Results} \label{sec:results}

\subsection{Theoretical model of joint improvisation} \label{sec:hypo}

Our first result is a mathematical model of Joint Improvisation. Our experimental observations of two HPs playing the mirror game suggests that their interaction in a JI condition is driven by three key factors: (i) their will  to synchronize each other's movement; (ii) the tendency of each player to exhibit some individual preferred movement features (or IMS); and (iii) the attempt each player makes to imitate the way the other moves (or mutual imitation). As shown in Fig~\ref{fig6}, we proposed to map each of these three factors onto a specific behavioral goal. 
In particular, {\it synchronization of joint movements} can be translated into the goal of minimizing the position mismatch between the balls moved by the two participants, which is related to the temporal correspondence (TC) between their positions.  {\it Spontaneous motion} preferences arise from the tendency of each participant to move according to his/her own IMS. Finally, {\it mutual imitation} can be  achieved by the participants  minimizing their velocity mismatch (velocity TC) during the mirror game. We captured each of these properties into a new mathematical model of interaction during JI formulated as the following optimal control problem.

\begin{figure}
\includegraphics[width=.95\textwidth]{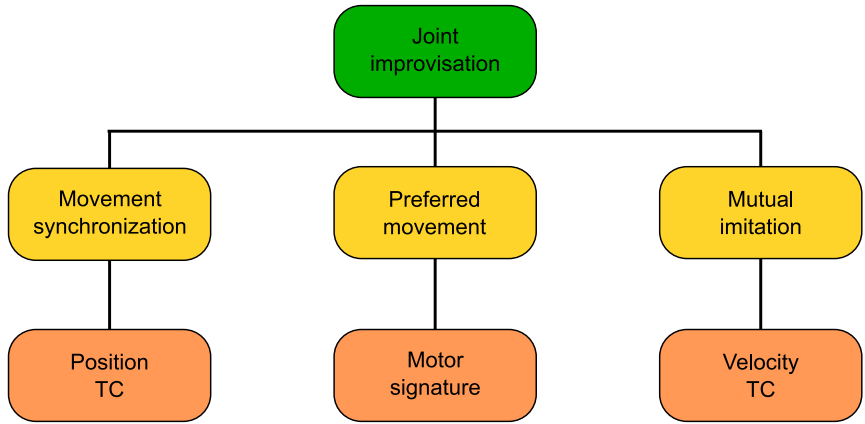}\centering
\caption{\label{fig6}{\bf Theoretical model of joint improvisation. } Three key factors in the emergence of JI (from left to right): 1. movement synchronization, corresponding to temporal correspondence (TC) between positions of the players' end effectors; 2. preferred movement, captured by their respective IMS; 3. mutual imitation, modeled by temporal correspondence between their velocities.}
\centering
\end{figure}

Specifically, following the approach of~\cite{chaosmc, chaocdc14, chao14}, we modeled the motion of each of the players using a nonlinear Haken-Kelso-Bunz (HKB) oscillator~\cite{hkb85} of the form
\begin{equation}\label{vplayer}
\ddot{x_i}+(\alpha_i \dot{x_i}^2+\beta_i x_i^2-\gamma_i)\dot{x_i}+\omega_i^2x_i=u_i,\qquad i=1,2
\end{equation}
where $x_i$ and $\dot{x}_i$ denote position and velocity of player $i$, $u_i$ is the coupling function through which player $i$ modulates its motion according to that of the other player, while $\alpha_i$, $\beta_i$, $\gamma_i$ and $\omega_i$ are intrinsic parameters determining the intrinsic properties of the player's motion, such as speed of reaction and settling time. 
We represented the coupling function $u_i$ as a nonliner control input that each player computes by minimizing the following cost function over each sampling period $T=t_{k+1}-t_k$ the whole trial duration is being split into. Namely, the problem is that of finding the inputs $u_i$ such that
\begin{equation}\label{minj}
\min_{u_i \in \mathbb{R}} J_i \left( x_1,x_2,\dot{x}_1,\dot{x}_2,t \right)
\end{equation}
where
\begin{equation}\label{vcost}
\begin{split}
J_i \left( x_1,x_2,\dot{x}_1,\dot{x}_2,t \right) &=\frac{\theta_{p,i}}{2}\underbrace{(x_1(t_{k+1})-x_2(t_{k+1}))^2}_{Position~ TC}+\frac{\theta_{\sigma,i}}{2}\int_{t_k}^{t_{k+1}}\underbrace{(\dot{x}_i(\tau)-{\sigma_i}(\tau))^2}_{Motor ~Signature}d\tau\\
&+\frac{\theta_{v,i}}{2}\int_{t_k}^{t_{k+1}}\underbrace{(\dot{x}_1(\tau)-\dot{x}_2(\tau))^2}_{Velocity~TC}d\tau
+\frac{\eta_i}{2}\int_{t_k}^{t_{k+1}}u_i(\tau)^2d\tau
\end{split}
\end{equation}
with $\theta_{p,i},\theta_{\sigma,i},\theta_{v,i},\eta_i>0$ being tunable control parameters satisfying the constraint $\theta_{p,i}+\theta_{\sigma,i}+\theta_{v,i}=1$. 
Here, $\sigma_i$ encodes the IMS of player $i$ as his/her velocity time series during solo trials. 
The cost function described in Eq~(\ref{vcost}) contains four terms. The first three terms correspond to each of the three factors characterizing JI shown in Fig~\ref{fig6}, while the fourth aims at minimizing the control effort over each sampling period. 

The tunable weights in the cost function can be used to change the balance between the four terms described above. The optimal control framework  allows to incorporate into the same cost function all key factors identified to govern the emergence of joint improvisation. In what follows we will refer to each of the players modelled by Eq from (\ref{vplayer}) to (\ref{vcost}) as a \emph{virtual player}. We will denote VPs as VP$_i$ where $i$ denotes that the model receives as an input the pre-recorded velocity profile $\sigma_i$ of the $i$-th HP playing solo. 

\subsection{Model testing and validation}
\label{sec:modtestval}
To test the effectiveness of the model we compared its predictions with the human benchmark described in Section \ref{sec:HP_HP}. A total of 9 different interactions in each of the 8 dyads were considered, since there were 3 solo trials available for each participant (corresponding to the reference velocity profiles $\sigma_i$ used in the model). Indeed, if we refer to the IMS of the $i$-th HP recorded in the $j$-th Solo trial of each dyad as $\sigma_{i,j}$, the 9 different interactions were obtained by feeding VP$_1$ with $\sigma_{1,h}$ and VP$_2$ with $\sigma_{2,k}$, where $h,k=1,2,3$ (Fig~\ref{fig7}). A table with the composition of all dyads is shown in Table \ref{tab:VPdyads} in Section \ref{S1_Text}.

\begin{figure}
\includegraphics[width=.95\textwidth]{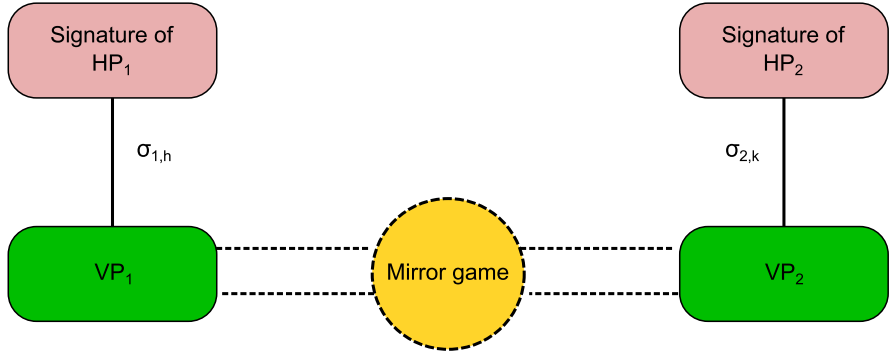}\centering
\caption{\label{fig7}{\bf Schematic diagram of VP-VP interaction.} VP$_1$ is fed with the motor signature of HP$_1$, while VP$_2$ with that of HP$_2$. We refer to the IMS of the $i$-th HP recorded in the $j$-th Solo trial of each dyad as $\sigma_{i,j}$. In this case $h,k=1,2,3$ give rise to 9 different combinations for each dyad. The JI session played by the virtual players resembles the one performed by the two HPs, whose respective motor signatures are fed to the VPs.}
\centering
\end{figure}

In so doing, for each dyad, each model equation was used to describe the kinematic behavior of a corresponding HP. The parameters of the model were set heuristically to the following values: $\alpha_1=\alpha_2=1$, $\beta_1=\beta_2=1$, $\gamma_1=\gamma_2=1$, $\omega_1=\omega_2=1$, $T=t_{k+1}-t_k=0.016$s. The weights $\theta_{p,i}$, $\theta_{\sigma,i}$ and $\theta_{v,i}$ were also set heuristically by trial-and-error in order to best match the experimental results (see Table \ref{tab:parset} in Section \ref{S1_Text} for further details on the values of the weights and how to interpret them).

\paragraph{Single JI trial comparison}
For the sake of clarity, we begin by showing a quantitative comparison of experimental data from a single HP-HP dyad with the corresponding data obtained by simulation of the model equations. In particular, we considered the third JI trial of Dyad 1. We denote by $\nu_i$ the velocity PDF of VP$_i$  and by $\mu_i$ the velocity PDF of its corresponding HP$_i$ from a simulated and experimental JI trial, respectively.

The values of the EMDs between velocity PDFs are given in Table~\ref{table1}. We found that the distance $\eta(\sigma_i,\nu_i)$ between the velocity PDFs of each virtual player (evaluated from the JI trial) and its reference motor signature $\sigma_i$ matches closely the distance $\eta(\sigma_i,\mu_i)$ between the corresponding HPs and their own signatures. Moreover, the distance $\eta(\nu_i,\nu_j)$ between  two VPs interacting with each other is quite close to that observed when the two HPs they model interact together in the mirror game, $\eta(\mu_i,\mu_j)$. This shows how, just like in the case of two humans (previously analyzed in Section \ref{sec:HP_HP}), the velocity PDFs of the two VPs move away from their respective signatures and get close to each other (while remaining close to the velocity PDFs of their human counterparts in the JI trial).

\begin{table}[b!]
\begin{center}
\caption{{\bf Evaluation of the model via EMDs between velocity PDFs for a single trial.} Here $\mu_i$ denotes the velocity profile of HP$_i$ during an experimental interaction with the other human player, $\nu_i$ that of the corresponding VP$_i$ playing with another VP {\em in-silico}, and $\sigma_i$ are the pre-recorded IMSs of the HPs.}\label{table1}
\begin{tabular}{|c|c|c|c|}
  \hline
  $\eta(\sigma_1,\mu_1$) & 0.102 & $\eta(\sigma_2,\mu_2$) & 0.052    \\ \hline
  $\eta(\sigma_1,\nu_1$) & 0.142 & $\eta(\sigma_2,\nu_2$) & 0.067    \\ \hline
  $\eta(\mu_1,\mu_2$)  & 0.030 & $\eta(\nu_1,\mu_1$) & 0.052     \\ \hline
  $\eta(\nu_1,\nu_2$)  & 0.021 & $\eta(\nu_2,\mu_2$) & 0.019    \\ \hline
\end{tabular}
\end{center}
\end{table}

Computation of the relative phases PDFs in HP-HP interactions and VP-VP interactions confirmed that they are close to each other, thus leading to the conclusion that, just like in the human scenario, neither of the two VPs turned out to be a leader during the JI trial. Indeed, if we denote with $\phi_{VP}$ and $\phi_{HP}$ the PDFs of the relative phase between the two VPs and the two HPs they model, respectively, the EMD between them was computed to be $\eta(\phi_{VP},\phi_{HP})=0.024$.

The previous findings show how the proposed model succeeds in capturing the main characteristics of the interaction between two human players improvising together, thus demonstrating its ability to reproduce {\em in-silico} the mirror game between two human subjects in a JI condition.
 
\paragraph{Matching results for the 8 dyads}
Next, we present results  for all the 8 experimental data-sets and show how our model is able to capture the experimental observations in terms of RMS position error, EMD between velocity PDFs and EMD between relative phase PDFs.

Figure~\ref{fig8} shows good agreement between RMS position errors observed in experiments (blue crosses, three trials) and those obtained from corresponding VP dyads (green boxes, nine trials). In particular, it is worth pointing out how the higher level of  synchronization between HPs in Dyad 5 (lower value of the RMS position error) was also captured in the VP-VP simulations.
 
\begin{figure}
\centering
\includegraphics{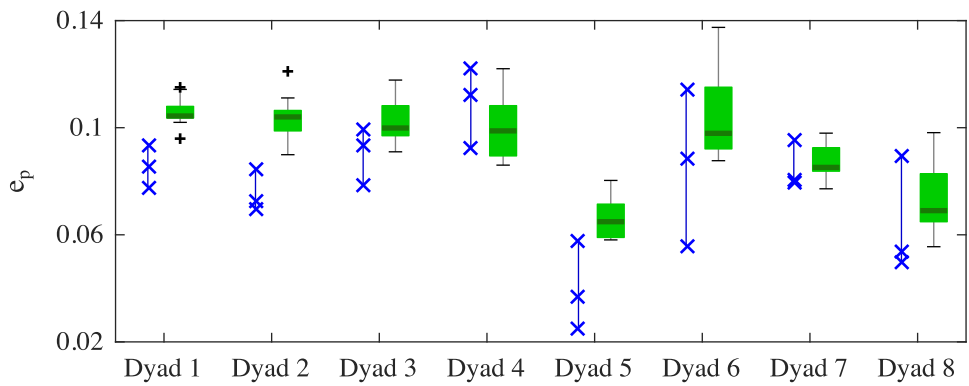}
\caption{\label{fig8}{\bf Matching in terms of temporal correspondence between HPs and between VPs.} Blue crosses (x) show the RMS position error from three JI trials of HPs, whilst box-plots depict the distributions of RMS position error from nine simulations of JI interaction between VPs (corresponding to nine combinations of the HPs' individual signatures). In particular: thick horizontal green lines indicate median of the distribution; central light green boxes show central 50\% of the data with lower and upper boundary lines being at the 25\% and 75\% quantiles of the data; two vertical whiskers estimate 99\% of the range of the data; black crosses (+) show outliers.}
\end{figure}

Figure~\ref{fig9} confirms the model ability to capture the behavioral plasticity of the IMS of the players during a typical game. Data from JI trials between HPs is shown in blue, while data corresponding to the VPs simulations is shown in green. We found that the EMD between velocity PDFs of two HPs was similar to that of the two corresponding VPs for all the dyads (Fig~\ref{fig9}A). Moreover, such similarity was found also in the relations between each player's motion and their signatures (Figs~\ref{fig9}B,C). It is worth pointing out how the higher value of $EMD$ between HP$_1$ and his/her corresponding signature for Dyad 7 and Dyad 8 was well captured by the model predictions (Fig~\ref{fig9}B).

\begin{figure}
\centering
\includegraphics{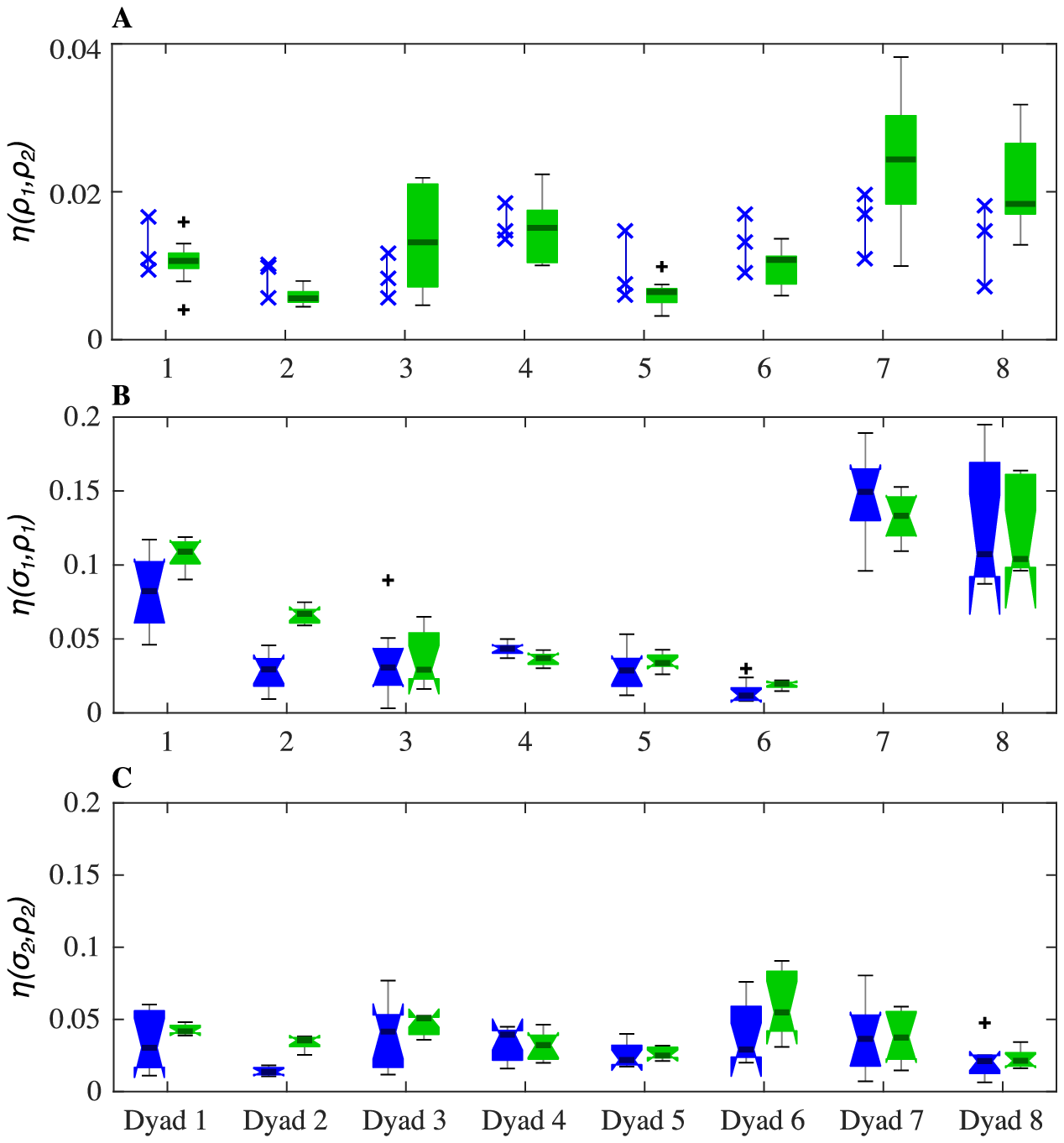}
\caption{\label{fig9}{\bf Matching in terms of relations between kinematics of HPs and VPs.} $\rho_i$ denotes the velocity PDF of the $i$-th player from a JI trial, that is $\mu_i$ for HP$_i$ and $\nu_i$ for VP$_i$.  A: degree of similarity between PDFs of velocities recorded during JI trials. Blue crosses show 3 trials for HPs and the green box-plots depict distributions of EMDs between velocities from 9 JI trials between VPs. B and C show how far the movements of the players in JI condition were shifted away from their motor signatures (blue for HPs, green for VPs). Notches on the box-plots indicate confidence intervals of the medians. Two medians are significantly different at the $p=0.05$ level if their notches do not overlap. Notches extending beyond the box indicate that the confidence intervals extend beyond central 50\% of the data points.}
\end{figure}

Figure~\ref{fig10} shows a quantitative comparison of the  PDFs of the relative phase computed from trials between HPs and VPs dyads. It is possible to appreciate how the EMD between relative phases in VP-VP trials, $\phi_{VP}$, and  HP-HP trials, $\phi_{HP}$, is low for all the eight dyads. This means that the absence of an emerging leader observed in a JI session between humans was also replicated in simulations between corresponding virtual players.

\begin{figure}
\centering
\includegraphics{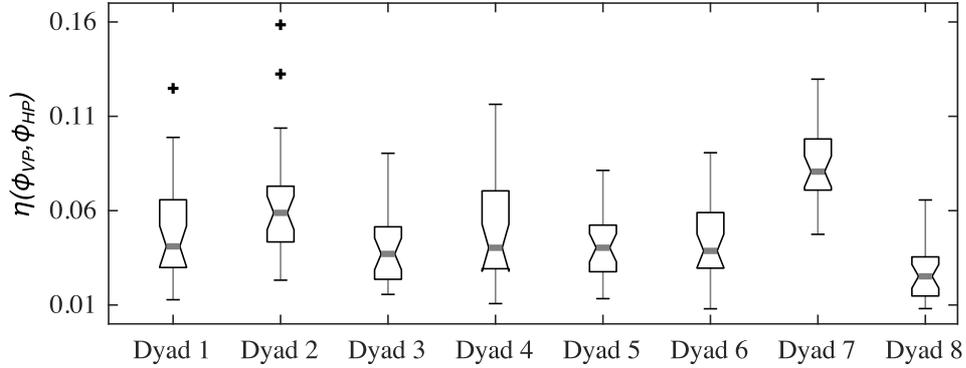}
\caption{\label{fig10}{\bf Matching in terms of relative phase between HPs and between VPs.} Box-plots illustrate distributions of EMDs between PDFs of relative phase from JI trials of HPs and VPs. Each box-plot corresponds to a single dyad and is constructed from 27 EMDs between three HP-HP relative phase PDFs and nine for the coupled VPs; black crosses (+) show outliers.}
\end{figure}

\subsection{A model-driven avatar} \label{sec:VP_HP}

Next, we used our theoretical model to drive a computer avatar able to play the mirror game with a human player in a JI condition. We employed the low-cost experimental Set-up 2 described in Section \ref{sec:expSetup} where the human player moves one of the two solid circles on the screen via a leap motion controller while the other is moved by the computer avatar.
 
The avatar computes the position of the circle it is moving by solving just one of Eq~(\ref{vplayer}), say for $i=1$, with $u_1$ obtained by solving the optimal control problem described in Eq~(\ref{vcost}). Now, $x_2$ and $\dot{x}_2$ indicate position and velocity of the HP (whose signature is denoted with $\sigma_{HP}$) interacting with the VP. In particular, velocity and position of such human player are estimated over each interval according to
\begin{equation*}\label{rvest}
\dot{x}_2(t)=\frac{x_2(t_k)-x_2(t_{k-1})}{T} \quad t\in[t_k, t_{k+1}]
\end{equation*}
and
\begin{equation*}\label{rpest}
x_2(t)=x_2(t_k)+\dot{x}_2(t) \left( t-t_k \right), \quad t\in[t_k, t_{k+1}]
\end{equation*}
Moreover, $\sigma_1$ in Eq~(\ref{vcost}) indicates the IMS of a different HP (which is fed to the VP) and is denoted with $\sigma_{VP}$. 

All the other model parameters were selected heuristically as follows: $\alpha_1=1$, $\beta_1=1$, $\gamma_1=1$, $\omega_1=1$, $\eta_1=10^{-4}$, $T=t_{k+1}-t_k=0.04$s, $\theta_{p,1}=0.2$, $\theta_{\sigma,1}=0.4$ and $\theta_{v,1}=0.4$. The initial position and velocity of the avatar were set to $0$.

The position time series recorded in the experiment for both players are shown in Fig~\ref{fig11}A. We found that the two main features of JI observed in the case of HP-HP interaction were replicated when replacing one of the two HPs with a VP. Indeed:
\begin{itemize}
\item both HP and VP move away from their own signatures and converge towards each other (Fig~\ref{fig11}B). In particular, $\eta(\nu,\sigma_{VP})=0.048$, $\eta(\mu,\sigma_{HP})=0.074$ and $\eta(\mu,\nu)=0.042$,
with $\mu$ and $\nu$ being the velocity PDFs obtained from the JI interaction between HP and VP, respectively;
\item the wide PDF of the relative phase between HP and VP indicates that there is no effective leader during the interaction (Fig~\ref{fig11}C).
\end{itemize}

\begin{figure}
\centering
\includegraphics{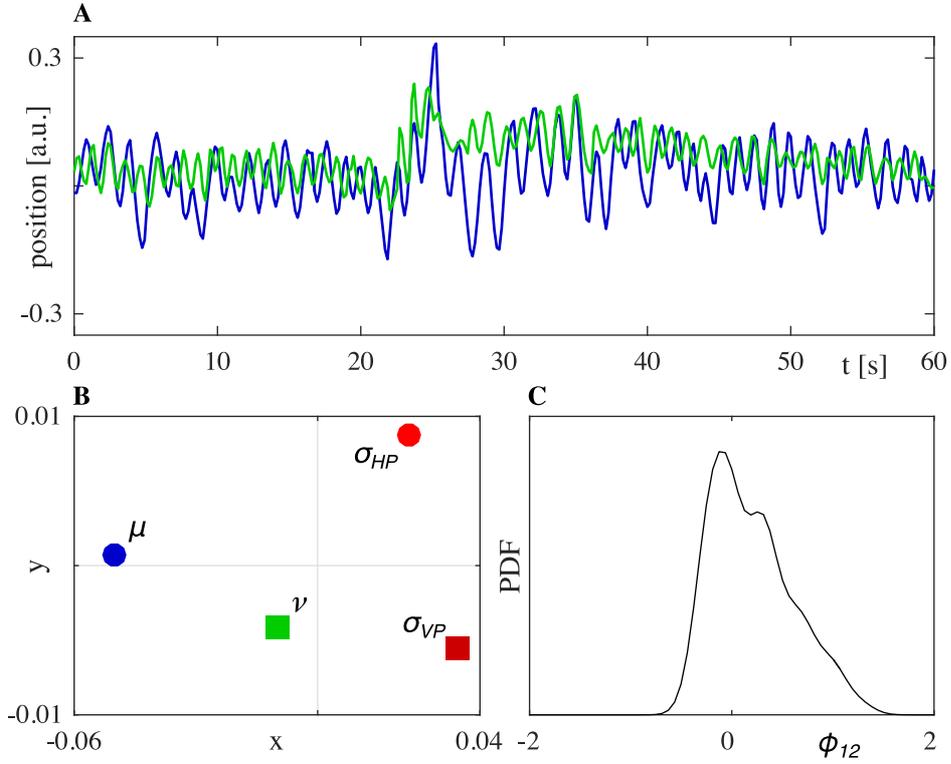}
\caption{\label{fig11}{\bf Interaction between a VP and a HP.} A: positions of HP (blue) and VP (green). B: relations between kinematics of motion of the players visualized by means of MDS. C: PDF of the relative phase between the two time-series of A.}
\end{figure}

Both results confirm that a computer avatar driven by our theoretical model is able to jointly improvise its motion in real-time with a human subject providing a new powerful tool for discovery and investigation of social interaction and movement coordination in the mirror game.

\subsection{Convergence analysis}
Finally, we confirmed via a theoretical analysis that the model we propose guarantees convergence between the players when either two coupled VPs are considered or when the model-driven avatar interacts with a human subject. Our main stability results can be listed as follows (see Section \ref{S1_Appendix} for a proof of the findings and further details):

\begin{itemize}
\item the solution to the minimization problem described in equations from~(\ref{vplayer}) to~(\ref{vcost}) ensures bounded position error between two VPs when {\em in-silico} experiments are considered;
\item the solution to the minimization problem described in equations from~(\ref{vplayer}) to~(\ref{vcost}) ensures bounded position error between HP and VP when the model-driven avatar interacts with a human subject;
\item if the nonlinear HKB dynamics of the VP end-effector described in Eq~(\ref{vplayer}) is substituted with a simpler linear model, achievement of the optimal solution to the minimization problem described in Eq~(\ref{minj}) and Eq~(\ref{vcost}) is guaranteed over each subinterval.
\end{itemize}

\section{Discussion}

In this paper we presented a new theoretical model of joint improvisation and highlighted the three key factors enhancing its emergence in joint action tasks within human dyads: movement synchronization, individual motor signature and mutual imitation. We then turned these ingredients into a mathematical model, based on the use of an HKB oscillator and an optimal control theoretic framework, to explain and reproduce the emergence of joint improvisation in human dyads.

Using both model predictions and experiments we demonstrated the applicability of our modelling approach to capture the emergence of joint improvisation between two human players in the mirror game and its capability to drive a computer avatar to produce jointly improvised movements with a human player.   

The availability of such an enhanced model-driven avatar provides a new fundamental tool to explore the important tenet in Social Psychology that behavioral similarity between people facilitates their interaction~\cite{Folkes1982,Marsh2009}. In particular, as recently proposed in \cite{piotr15},  the  similarity or dissimilarity between the IMSs of two individuals playing the mirror game can be an important factor affecting the level of their mutual interaction and coordination. Our proposed model can be used to test this hypothesis  both {\em in-silico} and via real-time experiments.

Future work will include looking into the joint improvisation among multiple human participants~\cite{fran15}. Finally we wish to highlight that the work presented in this paper opens the exciting possibility of performing {\em in-silico} experiments to assess how two human players (whose IMS have been recorded during solo trials) would interact when playing the mirror game in a JI condition. This can be useful for rehabilitation purposes such as those being explored as part of the EU project \emph{AlterEgo}~\cite{alter}.

\section{Supporting Information}

\subsection{Supporting Text}
\label{S1_Text}

In this section we show the parameter setting of the coupling weights for all the VP dyads (see Section \ref{sec:hypo}) and we describe all the available data.

\subsubsection{Parameter setting of the coupling weights for all the dyads}

\begin{table}[b!]
\begin{center}
\caption{ {\bf Parameter setting of the coupling weights $\theta_{p,i}$, $\theta_{\sigma,i}$ and $\theta_{v,i}$ for all dyads, with $\eta_1=\eta_2=10^{-4}$} }\label{tab:parset}
\begin{tabular}{|c|c|c|c|c|}
    \hline
    \multirow{2}{*}{Dyad1} &  VP1 & $\theta_{p,1}=0.10$& $\theta_{\sigma,1}=0.30$& $\theta_{v,1}=0.60$  \\
    \cline{2-5}
    & VP2 & $\theta_{p,2}=0.10$& $\theta_{\sigma,2}=0.55$& $\theta_{v,2}=0.35$  \\
    \hline
    \multirow{2}{*}{Dyad2} &  VP1 & $\theta_{p,1}=0.10$& $\theta_{\sigma,1}=0.35$& $\theta_{v,1}=0.55$  \\ \cline{2-5}
    & VP2 & $\theta_{p,2}=0.12$& $\theta_{\sigma,2}=0.45$& $\theta_{v,2}=0.43$  \\
    \hline
    \multirow{2}{*}{Dyad3} &  VP1 & $\theta_{p,1}=0.15$& $\theta_{\sigma,1}=0.30$& $\theta_{v,1}=0.55$  \\ \cline{2-5}
    & VP2 & $\theta_{p,2}=0.10$& $\theta_{\sigma,2}=0.35$& $\theta_{v,2}=0.55$  \\
    \hline
    \multirow{2}{*}{Dyad4} &  VP1 & $\theta_{p,1}=0.31$& $\theta_{\sigma,1}=0.38$& $\theta_{v,1}=0.31$  \\ \cline{2-5}
    & VP2 & $\theta_{p,2}=0.31$& $\theta_{\sigma,2}=0.38$& $\theta_{v,2}=0.31$  \\
    \hline
    \multirow{2}{*}{Dyad5} &  VP1 & $\theta_{p,1}=0.72$& $\theta_{\sigma,1}=0.22$& $\theta_{v,1}=0.06$  \\ \cline{2-5}
    & VP2 & $\theta_{p,2}=0.72$& $\theta_{\sigma,2}=0.22$& $\theta_{v,2}=0.06$  \\
    \hline
    \multirow{2}{*}{Dyad6} &  VP1 & $\theta_{p,1}=0.10$& $\theta_{\sigma,1}=0.60$& $\theta_{v,1}=0.30$ \\
    \cline{2-5}
    & VP2 & $\theta_{p,2}=0.10$& $\theta_{\sigma,2}=0.28$& $\theta_{v,2}=0.62$  \\
    \hline
    \multirow{2}{*}{Dyad7} &  VP1 & $\theta_{p,1}=0.10$& $\theta_{\sigma,1}=0.30$& $\theta_{v,1}=0.60$  \\ \cline{2-5}
    & VP2 & $\theta_{p,2}=0.10$& $\theta_{\sigma,2}=0.35$& $\theta_{v,2}=0.55$  \\
    \hline
    \multirow{2}{*}{Dyad8} &  VP1 & $\theta_{p,1}=0.10$& $\theta_{\sigma,1}=0.28$& $\theta_{v,1}=0.62$  \\ \cline{2-5}
    & VP2 & $\theta_{p,2}=0.10$& $\theta_{\sigma,2}=0.30$& $\theta_{v,2}=0.60$  \\
    \hline
\end{tabular}
\end{center}
\end{table}

We here comment on the interpretation of the weights $\theta_{p,i}$, $\theta_{\sigma,i}$ and $\theta_{v,i}$, $i\in\{1,2\}$ for each dyad of VPs (Table~\ref{tab:parset}). These weights were tuned by trial-and-error according to the analysis of HP-HP interactions so that the VP dyads could achieve the desired matching results with the human benchmark.

Since the weights can be interpreted according to our theoretical model, they provide further insights onto the JI interaction between HPs. Table~\ref{tab:parset} shows that, in order for the model to replicate some characteristics of the human JI interaction in dyads $1$, $2$, $3$, $6$, $7$ and $8$, the weights had to belong to the following intervals:
\begin{itemize}
\item $\theta_{p,i}\in[0.10,0.15], i\in\{1,2\}$;
\item $\theta_{\sigma,i}\in[0.28,0.60], i\in\{1,2\}$;
\item $\theta_{v,i}\in[0.30,0.60], i\in\{1,2\}$.
\end{itemize}
This indicates that the corresponding HPs paid more attention to individual preferences (their IMSs) and mutual imitation than to position mismatches. On the other hand:
\begin{itemize}
\item $\theta_{p,4}=0.31$, $\theta_{\sigma,4}=0.38$, $\theta_{v,4}=0.31$ lead to the conclusion that the corresponding HPs in Dyad 4 balanced all the three weights;
\item $\theta_{p,5}=0.72$, $\theta_{\sigma,5}=0.22$, $\theta_{v,5}=0.06$ lead to the conclusion that the corresponding HPs in Dyad 5 paid more attention to position error than to individual preferences and mutual imitation during their interaction.
\end{itemize}

\subsubsection{Available data}
\label{sec:availabledata}

Data structures for the experiments of HP-HP, VP-VP and VP-HP interactions are given as follows.
\begin{itemize}
\item HP-HP data has a format of a Matlab structure with 8 dyads fields each containing two players field. Each player field contains field solo with 3x1000 matrix containing solo recordings and field JI with 3x1000 matrix containing JI data. Sampling rate of the data is 100Hz.
\item VP-VP data has a format of a Matlab structure with 8 dyads fields each containing two players field. Each player field contains 9 trial fields. Each trial field contains vector $t$ of time stamps, vector $x$ of the simulated positions and variable $sig$ with the number indicating signature of the HP used in the simulations. Sampling rate is not constant. In order to have uniform sampling rate, in the paper data is interpolated using mean time-step and shape-preserving piecewise cubic interpolation.
\item VP-HP data has a format of a Matlab structure containing 4 fields. SigHP with signature of the HP, SigVP with signature of the VP, JIposHP position of the human player and JIposVP position of the VP. Fields sigVP and sigHP contain vectors $t$ of time stamps and $v$ of corresponding solo velocities, while fields JIposVP and JIposHP contain vectors $t$ of time stamps and $x$ of the corresponding positions when the HP and VP were playing together.
\end{itemize}
Each field contains vector $t$ of time stamps. 

In particular, refer to Table \ref{tab:avdata} for an explanation on how the data was collected and what it is made up of, to Table \ref{tab:VPdyads} for more information on the composition of dyads of VPs, to Table \ref{tab:table3} for more information on the data structure of a HP dyad, and to Table \ref{tab:table4} for more information on the data structure of a VP dyad.

\begin{table}[b!]
\begin{center}
\caption{ {\bf Available data}}\label{tab:avdata}
\begin{tabular}{|p{1.5cm}|p{10cm}|}
\hline
 HP-HP & Data was collected using experimental Set-up 1. Data was recorded from 8 dyads (16 participants in total). Available data has a format of a Matlab structure with 8 dyads fields each containing two players field. Each player field contains: field solo with 3x1000 matrix containing solo recordings and field JI with 3x1000 matrix containing JI data. Each recording has length of 60sec. Sampling rate of the data is 100Hz.\\ \hline
 
VP-VP & Simulations of the interactions between virtual players were run on a desktop computer using Matlab. Available data has a format of a Matlab structure with 8 dyads fields each containing two players field. Each player field contains 9 trial fields. Each trial field contains vector $t$ of time stamps, vector $x$ of the simulated positions and variable $sig$ with the number indicating signature of the HP used in the simulations. Sampling rate is not constant due to changes in duration of simulation steps. In order to have uniform sampling rate, data in the paper is interpolated using mean time-step and shape-preserving piecewise cubic interpolation.\\ \hline

VP-HP & Data was collected using experimental Set-up 2. Data was recorded for a single human participant interacting with a virtual player driven by a single motor signature. Available data has a format of a Matlab structure containing 4 fields. SigHP with signature of the HP, SigVP with signature of the VP, JIposHP position of the human player and JIposVP position of the VP. Fields sigVP and sigHP contain vectors $t$ of time stamps and $v$ of corresponding solo velocities, while fields JIposVP and JIposHP contain vectors $t$ of time stamps and $x$ of the corresponding positions when the HP and VP were playing together.\\ \hline
\end{tabular}
\end{center}
\end{table}

\begin{table}[b!]
\begin{center}
\caption{ {\bf Composition of dyads of VPs}}\label{tab:VPdyads}
\begin{tabular}{|c|c|c|c|}
\hline
 Dyad & Trial & Sig. $VP_1$ & Sig. $VP_2$ \\ \hline \hline
 1--8 & 1 & 1 & 1 \\ \cline{2-4}
        &  2 &  1 & 2 \\ \cline{2-4}
        &  3 &  1 & 3 \\ \cline{2-4}
         & 4 &  2 & 1 \\ \cline{2-4}
        &  5 &  2 & 2 \\ \cline{2-4}
        &  6 &  2 & 3 \\ \cline{2-4}
       &   7 &  3 & 1 \\ \cline{2-4}
       &  8 &  3 & 2 \\ \cline{2-4}
       &  9 &  3 & 3 \\ 
 \hline
\end{tabular}
\end{center}
\end{table}

\begin{table}[b!]
\begin{center}
\caption{ {\bf Data structure of a HP dyad}}\label{tab:table3}
\begin{tabular}{|c|c|c|c|c|}
\hline
 Dyad & Player & Condition & \multicolumn{2}{l|}{Matrix 3x1000}\\ \hline \hline
 1--8 & 1 & Solo & 1 &$x1_{s1}$\\ \cline{4-5}
    &   &          & 2 &$x1_{s2}$\\ \cline{4-5}
    &   &          & 3 &$x1_{s3}$\\ \cline{3-5}
     &  & JI      & 1 &$x1_{JI1}$\\ \cline{4-5}
    &   &          & 2 &$x1_{JI2}$\\ \cline{4-5}
    &   &          & 3 &$x1_{JI3}$\\ \cline{2-5}
   & 2 & Solo & 1 & $x2_{s1}$\\ \cline{4-5}
    &   &          & 2 & $x2_{s2}$\\ \cline{4-5}
    &   &          & 3 & $x2_{s3}$\\ \cline{3-5}
     &  & JI      & 1 & $x2_{JI1}$\\ \cline{4-5}
    &   &          & 2 & $x2_{JI2}$\\ \cline{4-5}
    &   &          & 3 & $x2_{JI3}$\\ \hline
\end{tabular}
\end{center}
\end{table}

\begin{table}[b!]
\begin{center}
\caption{ {\bf Data structure of a VP dyad}}\label{tab:table4}
\begin{tabular}{|c|c|c|c|c|c|}
\hline
 Dyad & Player & Trial & \multicolumn{3}{l|}{Fields: Signature, t, x}\\ \hline \hline
 1--8 & 1 & 1 & 1 & t1& $x1_{JI1}$\\ \cline{3-6}
    &   &  2 & 1 & t2& $x1_{JI2}$\\ \cline{3-6}
    &   &  3 & 1 & t3& $x1_{JI3}$\\ \cline{3-6}
     &  &  4 & 2 & t4&$x1_{JI4}$\\ \cline{3-6}
    &   &  5 & 2 & t5& $x1_{JI5}$\\ \cline{3-6}
    &   &  6 & 2 & t6&$x1_{JI6}$\\ \cline{3-6}
   &    &  7 & 3 & t7& $x1_{JI7}$\\ \cline{3-6}
    &   &  8 & 3 & t8& $x1_{JI8}$\\ \cline{3-6}
    &   &  9 & 3 & t9& $x1_{JI9}$\\ \cline{2-6}
   & 2 & 1 & 1 & t1& $x2_{JI1}$\\ \cline{3-6}
    &   &  2 & 2 & t2&$x2_{JI2}$\\ \cline{3-6}
    &   &  3 & 3 & t3& $x2_{JI3}$\\ \cline{3-6}
     &  &  4 & 1 & t4&$x2_{JI4}$\\ \cline{3-6}
    &   &  5 & 2 & t5& $x2_{JI5}$\\ \cline{3-6}
    &   &  6 & 3 & t6&$x2_{JI6}$\\ \cline{3-6}
   &    &  7 & 1 & t7&$x2_{JI7}$\\ \cline{3-6}
    &   &  8 & 2 & t8& $x2_{JI8}$\\ \cline{3-6}
    &   &  9 & 3 & t9& $x2_{JI9}$\\ 
 \hline
\end{tabular}
\end{center}
\end{table}

\clearpage

\subsection{Appendix}
\label{S1_Appendix}
We present here the theoretical analysis of the mathematical model we propose. More specifically we show that the solution to the minimization problem guarantees boundedness of the position error between the players when either two coupled VPs are considered or when the model-driven avatar interacts with a human subject.

\subsubsection{Model-driven avatar}

Let us recall that, over each sampling period $T=t_{k+1}-t_k$, the control input $u$ is obtained by solving the optimal control problem:
\begin{equation}\label{minjAPP}
\min_{u\in R} J
\end{equation}
where
\begin{equation}\label{vcostAPP}
\begin{split}
J&=\frac{\theta_p}{2}\underbrace{(x(t_{k+1})-\hat{r}_p(t_{k+1}))^2}_{Position~ TC}+\frac{\theta_{\sigma}}{2}\int_{t_k}^{t_{k+1}}\underbrace{(\dot{x}(\tau)-\sigma(\tau))^2}_{Motor ~Signature}d\tau\\
&+\frac{\theta_v}{2}\int_{t_k}^{t_{k+1}}\underbrace{(\dot{x}(\tau)-\hat{r}_v(\tau))^2}_{Velocity~TC}d\tau
+\frac{\eta}{2}\int_{t_k}^{t_{k+1}}u(\tau)^2d\tau
\end{split}
\end{equation}
with $\theta_p,\theta_{\sigma},\theta_v,\eta>0$ being tunable control parameters satisfying the constraint $\theta_p+\theta_{\sigma}+\theta_v=1$. Here, $x$ and $\dot{x}$ refer to position and velocity of the VP whose dynamics is modelled by:

\begin{equation}
\ddot{x}+(\alpha \dot{x}^2+\beta x^2-\gamma)\dot{x}+\omega^2 x=u
\end{equation}
Moreover, $\sigma$ refers to the IMS of a given HP in solo trials, while $\hat{r}_p$ and $\hat{r}_v$ represent the estimated position and velocity of the HP the VP is interacting with, respectively. In particular:
\begin{equation*}\label{rvest}
\hat{r}_v(t)=\frac{r_p(t_k)-r_p(t_{k-1})}{T} \quad t\in[t_k, t_{k+1}]
\end{equation*}
and
\begin{equation*}\label{rpest}
\hat{r}_p(t)=r_p(t_k)+\hat{r}_v(t)(t-t_k), \quad t\in[t_k, t_{k+1}]
\end{equation*}

\begin{thm}
\label{thm1}
The solution to the optimization problem described in Eq~(\ref{minjAPP}) ensures bounded position error between VP and HP.
\end{thm}

\begin{proof}
Let $J_0$ denote the value of the cost function described in Eq~(\ref{vcostAPP}) when $u\equiv0$. In addition, let $J^*$ and $x^*$ correspond to the optimal value of the cost function and the optimal position of the VP, respectively. It is clear that $J^*\leq J_0$ since $J^*$ is the minimum value of the cost function. According to Theorem $5.1$ in~\cite{chao14}, there exists a limit cycle in the HKB oscillator, and thus $x$ and $\dot{x}$ are bounded in $J_0$. Considering that $\hat{r}_p$, $\sigma$ and $\hat{r}_v$ are all bounded, we conclude that $J_0$ is bounded. It follows from the inequality
$$
\frac{\theta_p}{2}(x^*(t_{k+1})-\hat{r}_p(t_{k+1}))^2\leq J^* \leq J_0
$$
that the position error between VP and HP is bounded as well.  \hfill \end{proof}

\begin{cor}
\label{coro1}
If the nonlinear HKB dynamics of the VP end-effector is substituted with a simpler linear dynamics of the form
$$
\ddot{x}+a\dot{x}+bx=u
$$
achievement of the optimal solution to the minimization problem described in Eq~(\ref{minjAPP}) is guaranteed over each subinterval.
\end{cor}

\begin{proof}
According to the fundamental theorem of the calculus of variations, we need to examine the second variation of the given cost function in order to establish the optimum. From existing conclusions in~\cite{nai02}, the second variation of the cost function described in Eq~(\ref{vcostAPP}) in the Hamiltonian formalism is given by
\begin{equation*}
\begin{split}
\delta^2J&=\theta_p[\delta x(t_{k+1})]^2\\
&+\int_{t_k}^{t_{k+1}}\left(
                                  \begin{array}{cc}
                                    \delta X & \delta u \\
                                  \end{array}
                                \right)\left(
                                         \begin{array}{cc}
                                           H_{XX} & H_{Xu} \\
                                           H^T_{uX} & H_{uu} \\
                                         \end{array}
                                       \right)\left(
                                                \begin{array}{c}
                                                  \delta X \\
                                                  \delta u \\
                                                \end{array}
                                              \right)dt
\end{split}
\end{equation*}
where $\delta X=[\delta x~\delta\dot{x}]^T$ and $H$ is the Hamiltonian
\begin{equation*}
\begin{split}
H(X,u,\lambda)&=\frac{1}{2}\theta_{\sigma}(\dot{x}-\sigma)^2+\frac{1}{2}\theta_v(\dot{x}-\hat{r}_v)^2\\
&+\frac{1}{2}\eta u^2+\lambda^T\left(
          \begin{array}{c}
            \dot{x}\\
            -ay-bx+u \\
          \end{array}
        \right)
\end{split}
\end{equation*}
with $X=[x~\dot{x}]^T$ and $\lambda=[\lambda_1~\lambda_2]^T$.
Rewriting the linear system in matrix form we obtain
$$
\dot{X}=AX+Bu
$$
where
$$
A=\left(
    \begin{array}{cc}
      0 & 1 \\
      -b & -a \\
    \end{array}
  \right)
, \quad B=\left(
                         \begin{array}{c}
                           0 \\
                           1 \\
                         \end{array}
                       \right)
$$
Let $X=X^*+\delta X$ and $u=u^*+\delta u$, where $X^*$ and $u^*$ denote the optimal state and optimal control, respectively. Since $\dot{X}^*=AX^*+Bu^*$, we get
\begin{equation}\label{delX}
\delta \dot{X}=A \delta X+B\delta u
\end{equation}
Thus, it follows from $H_{Xu}=H_{uX}=[0~0]^T$, $H_{uu}=\eta>0$ and
$$
H_{XX}=\left(
         \begin{array}{cc}
           0 & 0 \\
           0 & \theta_{\sigma}+\theta_v \\
         \end{array}
       \right) \geq 0
$$
that
\begin{equation*}
\begin{split}
\delta^2J&=\theta_p[\delta x(t_{k+1})]^2\\
&+\int_{t_k}^{t_{k+1}}\delta X(t)^TH_{XX}\delta X(t)+\eta(\delta u(t))^2dt\\
&=\theta_p[\delta x(t_{k+1})]^2\\
&+\int_{t_k}^{t_{k+1}}(\theta_{\sigma}+\theta_v)(\delta \dot{x}(t))^2+\eta(\delta u(t))^2dt\\
&\geq0
\end{split}
\end{equation*}
Moreover, $\delta^2J=0$ is equivalent to $\delta x(t_{k+1})=0$, $\delta \dot{x}(t)=0$ and $\delta u(t)=0$ for all $t\in [t_k,t_{k+1}]$, which yields $\delta x(t)=\delta x(t_k)=0$ from Eq~(\ref{delX}). This corresponds to the optimal solution $X^*$ and the optimal control $u^*$. Therefore, we conclude that the optimal control ensures achievement of the minimum value of the cost function described in Eq~(\ref{vcostAPP}) over each sampling period.
\hfill
\end{proof}

\subsubsection{Two coupled VPs}
Let us recall that the model of two interacting VPs we propose consists of two coupled HKB oscillators:
\begin{equation}\label{couple_eq}
\left\{
  \begin{array}{ll}
   \ddot{x}_1+(\alpha_1 \dot{x}_1^2+\beta_1 x_1^2-\gamma_1)\dot{x}_1+\omega_1^2x_1=u_1 \\
   \ddot{x}_2+(\alpha_2 \dot{x}_2^2+\beta_2 x_2^2-\gamma_2)\dot{x}_2+\omega_2^2x_2=u_2
  \end{array}
\right.
\end{equation}
where $x_1$ and $x_2$ refer to the positions of the two virtual players VP$_1$ and VP$_2$, respectively. Analogously to the previous case, the control input for each HKB oscillator can be derived by making each VP solve the following optimal control problem
\begin{equation*}
\min_{u_i\in R} J_i, \quad i\in \{1,2\}
\end{equation*}
where
\begin{equation*}
\begin{split}
J_1&=\frac{\theta_{p,1}}{2}(x_1(t_{k+1})-{x}_2(t_{k+1}))^2\\
&+\frac{\theta_{\sigma,1}}{2}\int_{t_k}^{t_{k+1}}(\dot{x}_1(\tau)-\sigma_1(\tau))^2d\tau\\
&+\frac{\theta_{v,1}}{2}\int_{t_k}^{t_{k+1}}(\dot{x}_1(\tau)-\dot{x}_2(\tau))^2d\tau+\frac{\eta_1}{2}\int_{t_k}^{t_{k+1}}u_1(\tau)^2d\tau
\end{split}
\end{equation*}
and
\begin{equation*}
\begin{split}
J_2&=\frac{\theta_{p,2}}{2}(x_2(t_{k+1})-{x}_1(t_{k+1}))^2\\
&+\frac{\theta_{\sigma,2}}{2}\int_{t_k}^{t_{k+1}}(\dot{x}_2(\tau)-\sigma_2(\tau))^2d\tau\\
&+\frac{\theta_{v,2}}{2}\int_{t_k}^{t_{k+1}}(\dot{x}_2(\tau)-\dot{x}_1(\tau))^2d\tau+\frac{\eta_2}{2}\int_{t_k}^{t_{k+1}}u_2(\tau)^2d\tau
\end{split}
\end{equation*}
with $\theta_{p,i},\theta_{\sigma,i},\theta_{v,i},\eta_i>0, i \in \{1,2\}$ being tunable parameters satisfying the constraints $\theta_{p,i}+\theta_{\sigma,i}+\theta_{v,i}=1$.

In order to perform theoretical analysis for the nonlinearly coupled system described in Eq~(\ref{couple_eq}), we formulate the Hamiltonian for each of the two previous optimization problems as follows
\begin{equation*}
\begin{split}
H(X_i,u_i,\lambda_i)&=\frac{1}{2}\theta_{\sigma,i}(\dot{x}_i-\sigma_i)^2+\frac{1}{2}\theta_{v,i}(\dot{x}_1-\dot{x}_2)^2+\frac{1}{2}\eta_i u_i^2\\
&~~~+\lambda_i^T\left(
          \begin{array}{c}
            \dot{x}_i \\
            -(\alpha_i x_i^2+\beta_i \dot{x}_i^2-\gamma_i)\dot{x}_i-\omega_i^2x_i+u_i \\
          \end{array}
        \right)
\end{split}
\end{equation*}
where $X_i=[x_i~\dot{x}_i]^T$ and $\lambda_i=[\lambda_{i1}~\lambda_{i2}]^T$, $i\in\{1,2\}$. Applying the minimum principle~\cite{nai02}, we get the optimal open loop control inputs given by
$$
u_i=argmin_{u_i\in R}H(X_i,u_i,\lambda_i)=-\eta_i^{-1}\lambda_i^{T}\left(
                                                              \begin{array}{c}
                                                                0 \\
                                                                1 \\
                                                              \end{array}
                                                            \right)
=-\eta_i^{-1}\lambda_{i2}
$$
and the corresponding optimal state equations
\begin{equation}\label{x_i}
\begin{split}
\dot{X}_i&=\nabla_{\lambda_i}H(X_i,u_i,\lambda_i)\\
         &=\left(
          \begin{array}{c}
            \dot{x}_i \\
            -(\alpha_i x_i^{2}+\beta_i \dot{x}_i^{2}-\gamma_i)\dot{x}_i-\omega_i^2x_i-\eta_i^{-1}\lambda_{i2}\\
          \end{array}
           \right)
\end{split}
\end{equation}
with initial condition $X_i(t_k)=[x_i(t_k)~\dot{x}_i(t_k)]^T$. Also, the optimal costate equations can be written as
\begin{equation}\label{lam_1}
\begin{split}
\dot{\lambda}_1&=-\nabla_{X_1} H(X_1,u_1,\lambda_1)\\
\end{split}
\end{equation}
or equivalently as
$$
\left\{
  \begin{array}{ll}
    \dot{\lambda}_{11}&=\lambda_{12}(2\alpha_1 x_1\dot{x}_1+\omega_1^2) \\
    \dot{\lambda}_{12}&=\lambda_{12}(\alpha_1 x_1^{2}+3\beta_1 \dot{x}_1^{2}-\gamma_1)-\lambda_{11}-\theta_{\sigma,1}(\dot{x}_1-\sigma_1)\\
                      &-\theta_{v,1}(\dot{x}_1-\dot{x}_2) \\
  \end{array}
\right.
$$
and
\begin{equation}\label{lam_2}
\begin{split}
\dot{\lambda}_2&=-\nabla_{X_2} H(X_2,u_2,\lambda_2)
\end{split}
\end{equation}
or equivalently
$$
\left\{
  \begin{array}{ll}
    \dot{\lambda}_{21}&=\lambda_{22}(2\alpha_2 x_2\dot{x}_2+\omega_2^2)\\
    \dot{\lambda}_{22}&=\lambda_{22}(\alpha_2 x_2^{2}+3\beta_2 \dot{x}_2^{2}-\gamma_2)-\lambda_{21}-\theta_{\sigma,2}(\dot{x}_2-\sigma_2)\\ &-\theta_{v,2}(\dot{x}_2-\dot{x}_1)\\
  \end{array}
\right.
$$
with terminal conditions
$$
\lambda_1(t_{k+1})=\left(
                   \begin{array}{c}
                     \theta_{p,1}(x_1(t_{k+1})-{x}_2(t_{k+1})) \\
                     0 \\
                   \end{array}
                 \right)
$$
and
$$
\lambda_2(t_{k+1})=\left(
                   \begin{array}{c}
                     \theta_{p,2}(x_2(t_{k+1})-{x}_1(t_{k+1})) \\
                     0 \\
                   \end{array}
                 \right)
$$
Hence, the solution of the coupled model can be obtained by solving the above boundary value problem described by Eq~(\ref{x_i}), Eq~(\ref{lam_1}) and Eq~(\ref{lam_2}). Following the same proof strategy as in Theorem~\ref{thm1}, it is guaranteed that the position error between two VPs is bounded, and the proof is thus omitted to avoid redundancies. In particular, the emergence of joint improvisation movement is available by properly tuning the parameters $\theta_{p,i}$, $\theta_{\sigma,i}$ and $\theta_{v,i}$.

\section*{Acknowledgments}

This work was funded by the European Project AlterEgo FP7 ICT 2.9 - Cognitive Sciences and Robotics, Grant Number 600610. The authors wish to thank Prof. Benoit Bardy, Dr. Ludovic Marin and Dr. Robin Salesse at EUROMOV, University of Montpellier, France for all the insightful discussions and Ed Rooke at University of Bristol for collecting some of the experimental data that was used to validate the approach presented in this paper.

\section*{Author Contributions}

Conceived and designed the experiments: CZ, FA, PS, KTA, MdB. Performed the experiments: CZ, FA. Analyzed the data: CZ, FA, PS. Contributed reagents/materials/analysis tools: CZ, FA, PS. Wrote the paper: CZ, FA, PS, KTA, MdB.

%


\begin{thebibliography}{10}

\bibitem{Kelso1997}
Kelso JS.
\newblock Dynamic patterns: The self-organization of brain and behavior.
\newblock MIT press; 1997.

\bibitem{Schmidt2008}
Schmidt R, Richardson M.
\newblock Coordination: Neural, behavioral and social dynamics.
\newblock Volume: socialdyna. 2008;(4):1--53.

\bibitem{Schmidt2014}
Schmidt R, Nie L, Franco A, Richardson MJ.
\newblock Bodily synchronization underlying joke telling.
\newblock Frontiers in human neuroscience. 2014;8.

\bibitem{pnas14}
Dumas G, de~Guzman GC, Tognoli E, Kelso JS.
\newblock The human dynamic clamp as a paradigm for social interaction.
\newblock Proceedings of the National Academy of Sciences.
  2014;111(35):E3726--E3734.

\bibitem{john12}
Johnstone K.
\newblock Impro: Improvisation and the theatre.
\newblock Routledge; 2012.

\bibitem{saw01}
Sawyer RK.
\newblock Creating conversations: Improvisation in everyday discourse.
\newblock Hampton Press (NJ); 2001.

\bibitem{noy15}
Noy L, Levit-Binun N, Golland Y.
\newblock Being in the zone: physiological markers of togetherness in joint
  improvisation.
\newblock Frontiers in human neuroscience. 2015;9.

\bibitem{lak11}
Lakens D, Stel M.
\newblock If they move in sync, they must feel in sync: Movement synchrony
  leads to attributions of rapport and entitativity.
\newblock Social Cognition. 2011;29(1):1--14.

\bibitem{wil09}
Wiltermuth SS, Heath C.
\newblock Synchrony and cooperation.
\newblock Psychological science. 2009;20(1):1--5.

\bibitem{kel09}
Kelso JS, de~Guzman GC, Reveley C, Tognoli E.
\newblock Virtual partner interaction (VPI): exploring novel behaviors via
  coordination dynamics.
\newblock PLoS one. 2009;4(6):e5749.

\bibitem{kostrubiec15}
Kostrubiec V, Dumas G, Zanone P, Kelso J.
\newblock The Virtual Teacher (VT) Paradigm: Learning New Patterns of
  Interpersonal Coordination Using the Human Dynamic Clamp.
\newblock PLoS one. 2014;10(11):e0142029--e0142029.

\bibitem{noy11}
Noy L, Dekel E, Alon U.
\newblock The mirror game as a paradigm for studying the dynamics of two people
  improvising motion together.
\newblock Proceedings of the National Academy of Sciences.
  2011;108(52):20947--20952.

\bibitem{hart14}
Hart Y, Noy L, Feniger-Schaal R, Mayo AE, Alon U.
\newblock Individuality and togetherness in joint improvised motion.
\newblock PLoS one. 2014;9(2):e87213.

\bibitem{piotr}
S{\l}owi{\'n}ski P, Rooke E, di~Bernardo M, Tsaneva-Atanasova K.
\newblock Kinematic characteristics of motion in the mirror game.
\newblock In: 2014 IEEE International Conference on Systems, Man and
  Cybernetics (SMC). IEEE; 2014. p. 748--753.

\bibitem{piotr15}
S{\l}owi{\'n}ski P, Zhai C, Alderisio F, Salesse R, Gueugnon M, Marin L, et~al.
\newblock Dynamic similarity promotes interpersonal coordination in
  joint-action.
\newblock arXiv preprint arXiv:1507.00368. 2015.

\bibitem{levi01}
Levina E, Bickel P.
\newblock The earth mover's distance is the Mallows distance: Some insights
  from statistics.
\newblock In: Eighth IEEE International Conference on Computer Vision (ICCV).
  vol.~2. IEEE; 2001. p. 251--256.

\bibitem{Grinsted2004}
Grinsted A, Moore JC, Jevrejeva S.
\newblock Application of the cross wavelet transform and wavelet coherence to
  geophysical time series.
\newblock Nonlinear processes in geophysics. 2004;11(5/6):561--566.

\bibitem{MDS}
Borg I, Groenen PJ.
\newblock Modern multidimensional scaling: Theory and applications.
\newblock Springer Science \& Business Media; 2005.

\bibitem{chaosmc}
Zhai C, Alderisio F, Tsaneva-Atanasova K, di~Bernardo M.
\newblock A novel cognitive architecture for a human-like virtual player in the
  mirror game.
\newblock In: 2014 IEEE International Conference on Systems, Man and
  Cybernetics (SMC). IEEE; 2014. p. 754--759.

\bibitem{chaocdc14}
Zhai C, Alderisio F, Tsaneva-Atanasova K, di~Bernardo M.
\newblock Adaptive tracking control of a virtual player in the mirror game.
\newblock In: 53rd IEEE Conference on Decision and Control (CDC). IEEE; 2014. p. 7005--7010.

\bibitem{chao14}
Zhai C, Alderisio F, Tsaneva-Atanasova K, di~Bernardo M.
\newblock A model predictive approach to control the motion of a virtual player
  in the mirror game.
\newblock In: 54th IEEE Conference on Decision and Control (CDC). IEEE; 2015. p. 3175--3180.

\bibitem{hkb85}
Haken H, Kelso JS, Bunz H.
\newblock A theoretical model of phase transitions in human hand movements.
\newblock Biological cybernetics. 1985;51(5):347--356.

\bibitem{Folkes1982}
Folkes VS.
\newblock Forming relationships and the matching hypothesis.
\newblock Personality and Social Psychology Bulletin. 1982;8(4):631--636.

\bibitem{Marsh2009}
Marsh KL, Richardson MJ, Schmidt RC.
\newblock Social connection through joint action and interpersonal
  coordination.
\newblock Topics in Cognitive Science. 2009;1(2):320--339.

\bibitem{fran15}
Alderisio F, Bardy BG, di~Bernardo M.
\newblock Entrainment and Synchronization in Heterogeneous Networks of
  Haken-Kelso-Bunz (HKB) Oscillators.
\newblock arXiv preprint arXiv:1509.00753. 2015.

\bibitem{alter} AletrEgo Project.
\newblock Enhancing socio-motor interactions using information technology.
\newblock Available: http://www.euromov.eu/alterego/.

\bibitem{nai02}
Naidu DS.
\newblock Optimal control systems.
\newblock CRC press; 2002.

\end{thebibliography}

\end{document}